\documentclass{article}
\usepackage[utf8]{inputenc}
\usepackage[margin=1in]{geometry}
\usepackage{xcolor}
\usepackage{amsmath}
\usepackage{amsfonts,amssymb}
\usepackage[hidelinks]{hyperref}
\usepackage[nameinlink]{cleveref}

\usepackage{graphicx}
\usepackage{fancyhdr}
\usepackage{bm}
\usepackage{mathtools}

\usepackage{epstopdf,amsthm}
\usepackage{tikz}
\usepackage{mathrsfs}
\usepackage{pgfplots}
\pgfplotsset{compat=1.18} 
\newtheorem{conjecture}{Conjecture}
\newtheorem{remark}{Remark}
\newtheorem{theorem}{Theorem}
\Crefname{theorem}{Thm.}{Thms.}

\newcommand{\thistheoremname}{}
\newtheorem*{genericconj*}{\thistheoremname}
\newenvironment{namedconj*}[1]
  {\renewcommand{\thistheoremname}{#1}%
   \begin{genericconj*}}
  {\end{genericconj*}}

\newcommand{\x}{\mathbf{x}}
\newcommand{\z}{\mathbf{z}}
\newcommand{\aux}{\mathbf{a}}
\newcommand{\s}{\mathbf{s}}
\newcommand{\C}{\mathcal{C}}
\newcommand{\V}{\mathcal{V}}
\newcommand{\D}{\mathcal{D}}
\newcommand{\vv}{\mathbf{v}}
\newcommand{\bl}{\mathbf{b}_{l}}
\newcommand{\Zk}{\mathbb{Z}_{k}}
\newcommand{\UNSAT}{\texttt{unsat}}
\newlength{\Rpic}

\allowdisplaybreaks
\newcommand\numberthis{\addtocounter{equation}{1}\tag{\theequation}}
\newtheorem{lemma}{Lemma}
\usepackage{subcaption}

\usepackage{authblk}

\title{On the Emergence of Ergodic Dynamics in Unique Games}
\author[a,1]{Tuhin Sahai}
\author[b]{Abeynaya Gnanasekaran}
\affil[a]{SRI International, Menlo Park, CA 94025.}
\affil[b]{RTX Technology Research Center (RTRC), Berkeley, CA, USA 94705.}
\affil[1]{Corresponding author: Tuhin Sahai, tuhin.sahai@ gmail.com}
\begin{document}
\maketitle

\abstract{
The Unique Games Conjecture (UGC) constitutes a highly dynamic subarea within computational complexity theory, intricately linked to the outstanding P versus NP problem. Despite multiple insightful results in the past few years, a proof for the conjecture remains elusive. In this work, we construct a novel dynamical systems-based approach for studying unique games and, more generally, the field of computational complexity. We propose a family of dynamical systems whose equilibria correspond to solutions of unique games and prove that unsatisfiable instances lead to ergodic dynamics. Moreover, as the instance hardness increases, the weight of the invariant measure in the vicinity of the optimal assignments scales
polynomially, sub-exponentially, or exponentially depending on the value gap. We numerically reproduce a previously hypothesized hardness plot associated with the UGC. Our results indicate that the UGC is likely true, subject to our proposed conjectures that link dynamical systems theory with computational complexity. 
}

\section{Introduction}
The intersection of optimization theory and dynamical systems is rapidly emerging as a fertile area of research~\cite{sahai2020dynamical}. For instance, new insights related to the convergence of accelerated gradient methods were recently obtained by considering the continuous limits of these optimization schemes and the resulting dynamical systems~\cite{Cit:Nesterov_dyn,Cit:variational_nesterov}. Dynamical systems have been also used to construct state-of-the-art algorithms for graph clustering~\cite{sahai2010wave,Cit:sahai_hearing,zhu2022dynamic} by embedding graph structures into continuous spaces~\cite{Cit:belkin}. Flows on manifolds have been related to optimization problems such as least squares~\cite{Bro89},the traveling salesman problem~\cite{Cit:TSP}, graph matching~\cite{ZP08}, and linear programming~\cite{Bro91}. However, for the most part, the number of known connections between dynamical systems and optimization theories remain limited and anecdotal.

Previously, novel connections between dynamical systems theory and the iconic $K$-SAT problem have been identified~\cite{ercsey2011optimization,ercsey2012chaos}. In particular, the authors constructed a dynamical system such that the equilibria of the system of equations correspond to solutions of the $K$-SAT instance. They numerically observed that the transition from satisfiable (\texttt{sat})  to unsatisfiable (\UNSAT) $K$-SAT instances corresponds to the emergence of transient chaos.

In this work, our goal is to make fundamental connections between the theories of dynamical systems and computational complexity. In particular, we start with the unique games conjecture (UGC) that was originally posed to study the hardness of approximation of NP-hard problems. Generalizing~\cite{ercsey2011optimization}, we embed instances of unique games into corresponding dynamical systems, such that, their solutions form a one-to-one map to stable equilibria of the resultant dynamical system. Moreover, the dynamical systems have \emph{no attractors} that can trap trajectories. For unsatisfiable instances, the resulting dynamical systems are proven to be ergodic~\cite{cornfeld2012ergodic}. Here, ergodicity refers to the property of the dynamical system such that time--and space--averages converge to the same distribution (invariant measure)~\cite{cornfeld2012ergodic,katok1995introduction}. Just as in~\cite{ercsey2011optimization}, our dynamical systems are chaotic, however, the corresponding Lyapunov exponents can be made arbitrarily small (by picking higher values of exponents, see~\Cref{fig:fsle_average} in~\Cref{sec:sens_dep}). We present a family of \UNSAT~instances for the $K$-SAT problem whose unsatisfiability \emph{is computable in polynomial time} and yet gives rise to chaotic dynamics, demonstrating that chaos cannot be an appropriate indicator for hardness. These results run contrary to previous results and assertions~\cite{ercsey2011optimization,ercsey2012chaos}. 

Instead, we show that ergodicity provides a more appropriate ``score'' of problem hardness. Specifically, the weight of the ergodic invariant measure in the vicinity of the optimal assignments decreases polynomially, sub-exponentially, or exponentially as a function of the ``gap'' as hypothesized in~\cite{barak2018halfway}. Our work indicates that the UGC is \emph{likely true}. We propose conjectures that connect the two seemingly disparate areas of UGC and dynamical systems research.

Our paper is organized as follows, we start by constructing a new mapping from any unique game to a corresponding $K$-SAT instance. We then generalize~\cite{ercsey2011optimization} to a family of dynamical systems that can used to study the UGC. We analyze and present key properties of the resulting dynamical system and prove the emergence of ergodicity in \UNSAT~instances. We numerically show that as the \UNSAT~instances become ``harder'' (gap is decreased), the decay rate of the fraction of time the dynamical system spends in the vicinity of optimal assignments, transitions from polynomial scaling to sub-exponential and then to exponential, \emph{consistent with the UGC}. Our computed scalings, in parameter space, match a previously hypothesized plot~\cite{barak2018halfway}. Based on our results, we end the paper with three conjectures that connect the areas of dynamical systems theory and computational complexity.

\section{The Unique Games Conjecture (UGC)}
The UGC was formulated by Khot~\cite{khot2002power} in 2002 in an effort to shed light on the hardness of approximation of NP-hard problems such as label cover and mod-$2$ linear equations. The investigation of the conjecture and its implications has rapidly become one of the most active subareas of research in complexity theory, see~\cite{charikar2006near,trevisan2012khot,arora2010subexponential,o2012new} and references therein. Over the years, deep and surprising connections have been found between the conjecture and various problems. For example, if the conjecture~\cite{khot2002power} is true, it implies that the Goemans--Williamson algorithm for the MAX-CUT problem is optimal~\cite{khot2007optimal}. Similar approximation bounds have been found for the vertex cover~\cite{khot2008vertex}, betweenness~\cite{charikar2009every}, and Max 2-SAT~\cite{khot2007optimal} problems. The UGC is also germane to the settings of voting systems~\cite{isaksson2012maximally}, discrete Fourier transform~\cite{khot2005unique}, geometry~\cite{khot2005unique}, and surface areas computations of foam~\cite{feige2007understanding} to name a few.

The conjecture can be stated in terms of a system of linear equations of two variables on $\Zk$ (set of integers mod $k$), where the $l$-th equation is defined as follows~\cite{khot2002power,khot2016candidate,trevisan2012khot},
\begin{align}
\x_{i} =\x_{j} + \mathbf{b}_{l}\,\,\text{mod}\,\, k.
\label{eqn:2lin}
\end{align}
Here, $\x_{i}$, $\x_{j}$, and $\mathbf{b}_{l}$ all take integer values in the range $\left[0, k-1\right]$. Given predefined values for $\mathbf{b}$, the goal is to find values for $\x=\{\x_{1},\x_{2},\hdots,\x_{n}\}$, such that, the number of satisfied equations in Eqn.~\ref{eqn:2lin} are maximized. If there exists an assignment for $\x$ such that \emph{all} equations are satisfied, the instance is deemed \emph{satisfiable} ($\texttt{sat}$). Conversely, if only a fraction of the equations can be satisfied, the instance is deemed \emph{unsatisfiable} (\UNSAT). In the following, we will refer to the above system of equations as $2$-Lin-$k$. The unique games conjecture is posed in terms of the \UNSAT~instances as follows,
\begin{namedconj*}{Unique games conjecture}[Khot 2002]
For every $0<\epsilon <\frac{1}{2}$, given an instance of linear equations defined on $\mathbb{Z}_{k}$ (as shown in Eqn.~\ref{eqn:2lin}), in which a $1-\epsilon$ fraction of equations can be satisfied, then there exists a value of $k$ such that no polynomial time algorithm can find a solution that satisfies at least an $\epsilon$ fraction of equations.
\end{namedconj*}
The conjecture is striking, since it claims that, in general, for instances of  system~(\ref{eqn:2lin}) for which $1-\epsilon$ equations are satisfiable, \emph{no polynomial time algorithm} can guarantee satisfying an $\epsilon$ fraction of them. We note that there is an alternative (equivalent) form of the conjecture that is posed using promise problems~\cite{goldreich2006promise}. To analyze the conjecture using the lens of dynamical systems theory, we now construct dynamical systems such that their attractors correspond to solutions of instances of linear equations defined over $\Zk$. We build these dynamical systems by embedding the linear equation instances into variable size SAT problems as described below.

\subsection{Equivalent SAT problem construction}
Satisfiability problems are defined using Boolean variables and logical operations (we denote the \emph{or} and \emph{and} operators by $\vee$ and $\wedge$ respectively)~\cite{cook1971complexity}. We focus on the standard conjunctive normal form (CNF) formulae which are conjunctions of clauses and the clauses are, in turn, disjunctions of variables.

Now consider a system of $n_{\text{eq}}$ linear equations defined over $\Zk$ (in Eqn.~\ref{eqn:2lin}), where the dimension of $\x$ is $n_{\x}$. We map each $\x_{i}$ variable to a set of $k$ spin variables (that are equivalent to Boolean variables), $S_{\x_{i}}=\{\s_{k(i-1)}, \s_{k(i-1) + 1},\s_{k(i-1) + 2},\hdots, \s_{k(i-1) + k-1}\}$, where $\s_l\in\{-1,1\}$.
Here, 
$\s_l=1$ corresponds to a \texttt{true} values and, conversely, $\s_l=-1$ corresponds to \texttt{false} values~\cite{khot2002power}.  Note that, although, the mapping from 2-Lin-$k$ equations to 2-SAT in~\cite{khot2002power} gives rise to SAT instances of smaller size, it results in stiffer dynamical systems, making computations challenging.

Since $\x_i$ can only take one value, 
only one element of $S_{\x_{i}}$ 
can be assigned a value of $1$ (\texttt{true}). The rest of the elements must all be $-1$ (\texttt{false}). To impose the above constraint, we include the following set of clauses for each $\x_{i}$,
\begin{align}
(\s_{k(i-1)}&\vee \s_{k(i-1) + 1}\hdots\vee \s_{ki -1})  \wedge  \nonumber\\
&(\wedge_{q_1<q_2}(\neg \s_{k(i-1) + q_1} \wedge \neg \s_{k(i-1)+q_2} )),
\label{eqn:variable_clause}
\end{align}
where the variables in the above  formula correspond to their equivalent Boolean assignments ($\{0,1\}$ values) and $\neg$ denotes the \emph{not} operation. Since each $\x_i$ maps to $k$ variables $\s_{k(i-1) + j}$,  where $0\leq j\leq k-1$, this gives rise to $M_{v} = 1 + \frac{k(k-1)}{2}$ clauses.

The 2-Lin-$k$ system of equations (Eqn.~\ref{eqn:2lin}) are bijective mappings~\cite{trevisan2012khot} and each equation corresponds to the following disjunctive normal form (DNF) formula, 
\begin{align}
((\s_{k(i-1)}&\wedge \s_{k(j-1) + \bl\,\text{mod}\,k}) \vee (\s_{k(i-1) + 1}\wedge \s_{k(j-1) + \bl+1\,\text{mod}\,k}) \hdots\nonumber\\
&\hdots\vee (\s_{ki-1}\wedge \s_{k(j-1) + \bl+k-1\,\text{mod}\,k})),
\label{eqn:dnf}
\end{align}
which can be converted to a CNF form that preserves satisfiability. A na\"ive conversion from DNF to CNF can cause an exponential increase in the number of clauses. However, we use the transformation outlined in~\cite{jackson2004optimality} that results in a linear increase in the number of clauses at the cost of introducing new variables $\z$. We note that, for every equation in Eqn.~\ref{eqn:2lin}, the final number of clauses in the CNF formula is $1+2k$ at the cost of introducing $k$ additional variables. For further details regarding the introduction of the variables and the resulting CNF formulae, we refer the reader to~\Cref{sec:mapping_conv}. 

One can verify that the total number of resulting variables and clauses are $N=(n_{\x} + n_{\text{eq}})k$ and $M=n_{\x}\frac{k(k-1)}{2} + n_{\text{eq}}(1+2k)$, respectively. An interesting aspect of the resulting SAT instance is that the term $(\s_{k(i-1)}\vee \s_{k(i-1) + 1}\vee...\vee \s_{ki-1})$ in Eqn.~\ref{eqn:variable_clause} is redundant with Eqn.~\ref{clause} in~\Cref{sec:mapping_conv} and can be ignored~\cite{khot2002power}. 

\section{Dynamical Systems for the UGC}
For SAT equations, we concatenate all the variables into one state vector $\vv = \left[\x,\z\right]$ (where $\z$ are the additional variables introduced during the DNF to CNF conversion outlined previously). Now, let $\s_{p} = \left[-1,1\right]$, i.e., $\s_{p}$ can take real values between $-1$ and $1$. 
Generalizing the dynamical system for satisfiability problems constructed in~\cite{ercsey2011optimization}, one can define $c_{mp}$ and ${K_m}$ as follows,
\begin{align*}
c_{mp} =
\begin{cases}
-1, & \text{if}\ \s_{p}\,\, \text{appears in negated form in $m$-th clause,}  \\
1, & \text{if}\ \s_{p}\,\, \text{appears in direct form in $m$-th clause,}  \\
0, & \text{otherwise},
\end{cases}
\end{align*}

\begin{align*}
K_m(\s) = 2^{-k} \prod_{p=1}^{N} (1-c_{mp}\s_{p}) \quad \forall m=1,2,\hdots,M.
\end{align*}
Note that $K_m(\s)=0$, if and only if the $m$-th clause is satisfied, i.e., $c_{mp}\s_{p}=1$ for at least one variable that appears in clause $m$. Following the steps in~\cite{ercsey2011optimization},  we define an energy function of the form $\V(\s,\aux)=\sum_{m=1}^{M}a_{m}K_{m}(\s)^2$, such that, $\V(\s^{*},\aux)=0$ only at solutions $\s^{*}$ of the satisfiability problem (and consequently the 2-Lin-$k$ system). The \emph{auxiliary} variables $a_{m}\in (0,\infty)$, prevent  non-solution attractors from trapping the search dynamics, and can be viewed as a form of Lagrange multipliers (see~\cite{ercsey2011optimization} for more information). We now generalize the system to,
\begin{align}
\frac{d\s_{l}}{dt} &= -(\nabla_s\V(\s,\aux)),  \nonumber\\
 &= \sum_{m=1}^{M} 2a_m c_{ml}K_{ml}(\s)K_m(\s) \quad \forall l=1,2,\hdots,N,\nonumber  \\
\frac{da_m}{dt} &= a_m(K_m)^{\alpha} \quad \forall m=1,2,\hdots,M,
\label{eqn:dyn}
\end{align}
where,
\begin{align*}
K_{ml}(\s) = 2^{-k} \prod_{\substack{p=1\\p\neq l}}^{N} (1-c_{mp}\s_{p}).
\end{align*}
Note that the dynamical system in~\cite{ercsey2011optimization}, corresponds to $\alpha=1$ in the above system of equations. In~\cite{ercsey2011optimization}, the authors find that as the constraint density of the $K$-SAT problem increases, the trajectories of the dynamical system display transient chaos with fractal basin boundaries~\cite{ercsey2011optimization,ercsey2012chaos}. The authors determine the existence of transient chaos by demonstrating the emergence of positive values for finite size Lyapunov exponents (FSLE)~\cite{FSLE,strogatz2018nonlinear} in certain parameter regimes. In this work, we will use the same numerical indicator for chaos with the caveat that, typically, other criteria are required to elucidate the route to chaos~\cite{guckenheimer1983local}.  Moreover, the authors claim that the emergence of chaos corresponds to the well known phase transitions in the $K$-SAT problem~\cite{ercsey2011optimization,monasson1999determining}. However, we note that \emph{any} problem instance that is unsatisfiable, will map to a chaotic dynamical system (see~\Cref{thm: lyapunov} in~\Cref{sec:sens_dep}). One can, therefore, construct examples that refute the proposed correspondence between NP-hardness and chaos. For example, any $3$-SAT instance that contains the following clause structure $(y_{i_1}\vee y_{i_2} \vee y_{i_3})\wedge(y_{i_1}\vee y_{i_2} \vee \neg y_{i_3})\wedge (y_{i_1}\vee \neg y_{i_2} \vee  y_{i_3})\wedge (\neg y_{i_1}\vee y_{i_2} \vee y_{i_3})\wedge(y_{i_1}\vee \neg y_{i_2} \vee \neg y_{i_3})\wedge(\neg y_{i_1}\vee  y_{i_2} \vee \neg y_{i_3})\wedge(\neg y_{i_1}\vee \neg y_{i_2} \vee   y_{i_3})\wedge(\neg y_{i_1}\vee \neg y_{i_2} \vee \neg  y_{i_3})$ \emph{must} be \UNSAT~and result in chaotic dynamical systems (as shown in~\Cref{thm: lyapunov}). For \emph{any} Boolean formula, one can search for this structure in polynomial time. Thus, we assert that the transient chaos is \emph{not} an appropriate marker for the hardness of an instance. We also note that this structure can trivially be generalized to the $K$-SAT setting.

By introducing the exponent $\alpha$, in Eqns.~\ref{eqn:dyn}, we can significantly change the magnitude of the FSLE without modifying the underlying hardness of the instance (see remark~\ref{remark:fsle}and associated~\Cref{fig:fsle_average} in~\Cref{sec:sens_dep}). Our modified dynamics preserves the desirable properties such as (a) spurious attractors that do not correspond to solutions of the SAT problem are absent, (b) solutions of the $K$-SAT correspond to \emph{stable} equilibria of the corresponding dynamical system, and (c) the values of $\s_{l}$ remain bounded within $\C=[-1,1]^N$. Example trajectories are shown in Fig.~\ref{fig:trajex} (for more details see~\cite{ercsey2011optimization}). Thus, given the equivalence between instances of the 2-Lin-$k$ system of equations (Eqn.~\ref{eqn:2lin}) and the corresponding $K$-SAT (Eqn.~\ref{eqn:variable_clause}), important connections for solutions of Eqn.~\ref{eqn:2lin} can be drawn using properties of the dynamical systems such as equilibria, stability, etc.


\begin{figure}
\centering
\resizebox{0.6\textwidth}{!}{
\input{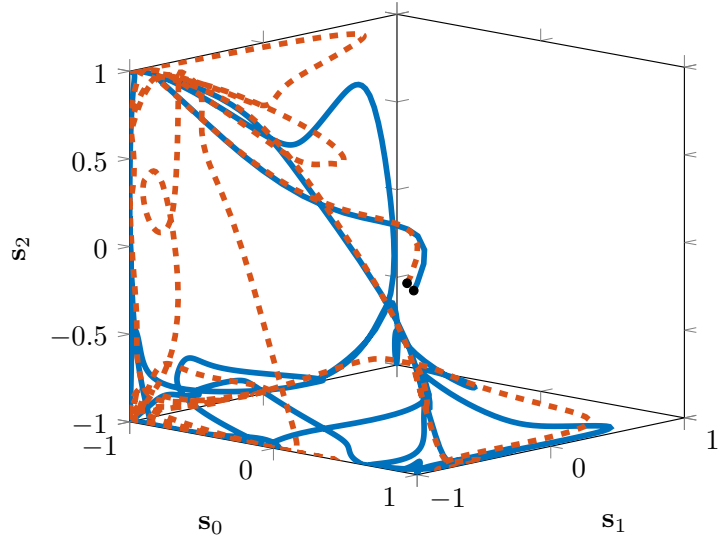}
}
\caption{Example trajectories evolving on $\C=[-1,1]^{N}$ for a simple UNSAT instance of the 2-Lin-$k$ system of equation.}
\label{fig:trajex}
\end{figure}
In the following, we show that for unsatisfiable instances of the UGC, the underlying dynamical system becomes ergodic, and the associated invariant measure~\cite{katok1995introduction} is a \emph{fundamental} quantity that captures the \emph{hardness}. For example, as the size of the instance grows, the weight of the invariant measure in the proximity of the optimal solution decreases exponentially for NP-hard problems and polynomially for tractable ones. This scaling property is found to depend on the alphabet size, the primary parameter that changes the underlying hardness. We start by proving that unsatisfiable instances of $2$-Lin-$k$ equations give rise to dynamics with sensitive dependence to initial conditions due to amplification of perturbations by the auxiliary variables. We then \emph{prove that these unsatisfiable instances of $2$-Lin-$k$ equations must result in ergodic dynamics} and numerically show that the scaling of the invariant measure is consistent with the unique games conjecture and accurately captures the various hardness regimes~\cite{arora2010subexponential,arora2010subexponential}.

Note that in~\cite{ercsey2011optimization}, the authors hypothesize that the emergence of chaos with increasing clause density is due to homoclinic and heteroclinic intersections~\cite{guckenheimer1983local}. However, we now show that the sensitive dependence to initial conditions (manifested by positive FSLE) occurs in all \UNSAT~instances of the $K$-SAT problem due to the dynamics of the auxiliary variables and their \emph{exponential growth dynamics}.
\begin{theorem}
All unsatisfiable (\UNSAT) instances of the $K$-SAT and $2$-Lin-$k$ equations result in dynamical systems (described by Eqn.~\ref{eqn:dyn}) that display sensitive dependence to initial conditions. This is a result of the exponential growth dynamics of the auxiliary variables that amplifies any initial perturbations in the values of the spin variables.
\end{theorem}
\begin{proof}
See~\Cref{sec:sens_dep} for a proof of the above theorem. 
\end{proof}
\begin{figure}
\centering
\includegraphics[width=.6\linewidth]{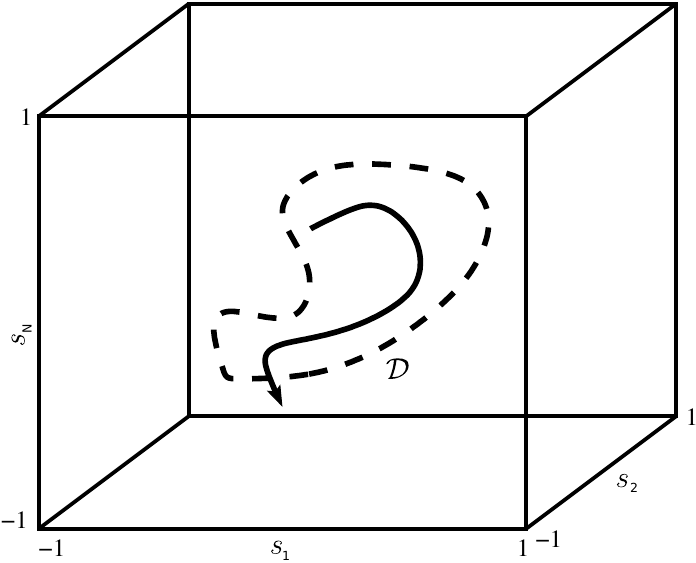}
\caption{Given any \UNSAT~instance of Eqn.~\ref{eqn:2lin}, any trajectory starting inside an arbitrary domain $\D\subseteq\C$ will eventually escape it.}
\label{fig:escaping}
\end{figure}
\begin{theorem}
All unsatisfiable instances of 2-Lin-$k$ system of equations (described by Eqn.~\ref{eqn:2lin}) result in ergodic dynamical systems. The emergence of ergodicity is a consequence of the special property of the dynamical system such that no trajectory can remain indefinitely constrained within any closed subset of $\C$.
\end{theorem}
\begin{proof}
The core of our argument rests on the fact that a trajectory will eventually escape any arbitrary closed domain $\D$ (see Fig.~\ref{fig:escaping}) within $\C$~\cite{ercsey2011optimization}. As demonstrated in~\cite{ercsey2011optimization}, a domain $\mathcal{D}$ that contains a non-solution fixed point cannot capture the dynamics of a trajectory for any formula. Note that \emph{any} unsatisfied clause leads to an exponential growth of the corresponding $a_m$ coefficient and has the following form,
\begin{align}
    \frac{d\s_l}{dt} = \displaystyle\sum_{m=1}^{M}2a_{m}(0)c_{ml}K_{ml}K_{m}e^{\int_{0}^{t} K_{m}dt}
\end{align}
Thus, if a trajectory remains bounded in $\mathcal{D}$, some constraints will remain unsatisfied (by virtue of not containing solutions), leading to exponential growth of the $a_m$ values. Even if the exponentials cancel \emph{exactly}, any small computational perturbation (from round-off or numerical errors), will blow up, leading the trajectory to escape from $\mathcal{D}$. For more details see~\cite{ercsey2011optimization}. Now consider an infinitesimal ball $\mathcal{B}_{\mathcal{P}}$ around any point $\mathcal{P}\in \mathcal{C}$. Then, as shown, a trajectory from \emph{any} initial condition \emph{must} exit $\mathcal{C}\setminus\mathcal{B}_{\mathcal{P}}$ (enter $\mathcal{B}_{\mathcal{P}}$). Using the same argument, trajectory must also eventually exit $\mathcal{B}_{\mathcal{P}}$. Thus, the trajectory must enter the neighborhood of all points in $\mathcal{C}$ an \emph{infinite number of times}. The invariant measure is given by, $\mu(\mathcal{B}_{\mathcal{P}}) = \mathcal{F}(\mathcal{B}_{\mathcal{P}})$, where $\mathcal{F}(\mathcal{B}_{\mathcal{P}})$ is the fraction of time the trajectory spends in $\mathcal{B}_{\mathcal{P}}$ (time average). The dynamical system is consequently ergodic with an invariant measure $\mu$. 
\end{proof}
\begin{remark}
Combining the implications of the two above theorems (positive Lyapunov exponents and ergodicity), we note that the dynamical system is in fact a $K$-system~\cite{frigg2011ergodic,katok1995introduction}. 
\end{remark}
It is interesting to note that if the invariant measure places fundamental limitations on the \emph{best} possible algorithm for solving $2$-Lin-$k$ systems, it also places computational limits on simulations of physical systems, i.e., the ability to efficiently compute solutions to ergodic systems with increasing dimensionality and stiffness would imply that the UGC is not true.
\section{Results}
\subsection{Generating UGC instances}
For computational tests, we now describe our technique for constructing instances of 2-Lin-$k$ equations with a user-prescribed fraction of satisfiable equations. As depicted in Fig.~\ref{fig:eqn_gen}, our approach can be visualized using an $n$-sided polygon whose vertices or nodes denote the variables in the 2-Lin-$k$ system and the lines that connect the nodes depict the equations that relate one variable to another (of the form Eqn.~\ref{eqn:2lin}). We start by constructing the equations that relate $\x_{i}$ to $\x_{i+1}$ and $\x_{1}$ to $\x_{n}$. It is straightforward to check that out of these $n$ equations, $n-1$ must always be satisfied. We then start adding the equations that correspond to `diagonals' of the polygon. One can pick values for $\mathbf{b}_{l}$  such that the additional equations either add to the $\texttt{sat}$ or \UNSAT~lists, thereby, providing the user the ability to construct a UGC instance with a predefined value for $\epsilon$. Note here, for any specific equation, one can find an assignment to satisfy it. However, it will lead to the constraint violation of another equation on the list, leading to net zero gain in the value of $\epsilon$.
\begin{figure}
\centering
\begin{tikzpicture}
  [inner sep=1mm,
  minicirc/.style={circle,draw=black!10,fill=black!10,thick}]

  \node (circ1) at ( 60:\Rpic) [minicirc] {$\x_{1}$};
  \node (circ2) at (120:\Rpic) [minicirc] {$\x_{2}$};
  \node (circ3) at (180:\Rpic) [minicirc] {$\x_{3}$};
  \node (circ4) at (240:\Rpic) [minicirc] {$\x_{4}$};
  \node (circ5) at (300:\Rpic) [minicirc] {$\x_{5}$};
  \node (circ6) at (360:\Rpic) [minicirc] {$\x_{6}$};

  \draw [ultra thick] (circ1) to (circ2) to (circ3)
  to (circ4) to (circ5) to (circ6);
  \draw [ultra thick,dashed] (circ1) to (circ6);
  \draw [ultra thick,dashed] (circ2) to (circ5);
\end{tikzpicture}
\caption{Method for generating instances of the unique games with user prescribed number of equations. Nodes are variables in the 2-Lin-$k$ system, and lines denote the equations that relate one variable to another. Every solid line corresponds to a satisfiable equation, whereas, every dotted line corresponds to an unsatisfiable equation. Note that one can force the ``unsat'' equation to be satisfied, at the expense of converting another ``sat'' equation to ``unsat''.}
\label{fig:eqn_gen}
\end{figure}
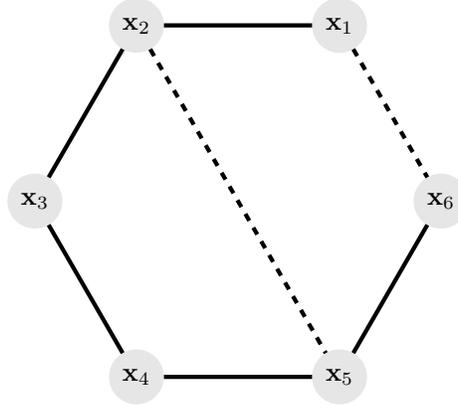
\subsection{Computational results}
To analyze the unique games conjecture using the dynamical system above, we generate random instances of $2$-Lin-$k$ equations and vary the alphabet $k$ and system size given by $n_{\x}$. For prescribed values of $n_{\x}$ and $k$, we generate instances such that the maximum number of satisfiable equations (given by $1-\epsilon$) is known a priori. We then generate the corresponding dynamical system and integrate it numerically starting from a random initial condition. Note that since these instances are unsatisfiable, the corresponding trajectories are ergodic in $\mathcal{C}$. 

For $2$-Lin-$k$ instances generated using our approach, since we a priori know $1-\epsilon$, we compute the percentage of time the trajectories spend in the vicinity of assignments that satisfy \emph{at least} any $\delta$ fraction of the equations, and study the impact of varying $k$. In Fig.~\ref{fig:transition}, we present the \emph{computational scaling} of the time spent by the trajectories in the ``vicinity'' of $\delta$ equations satisfied, denoted by $\mathcal{Y(\delta)}$, as a function of $k$. Here, we consider the vicinity of a particular assignment to be the $L_1$-ball of radius $0.1$ around it. In other words, if a trajectory enters this ball, we assume it is spending time satisfying the associated $\delta$ fraction of equations. 
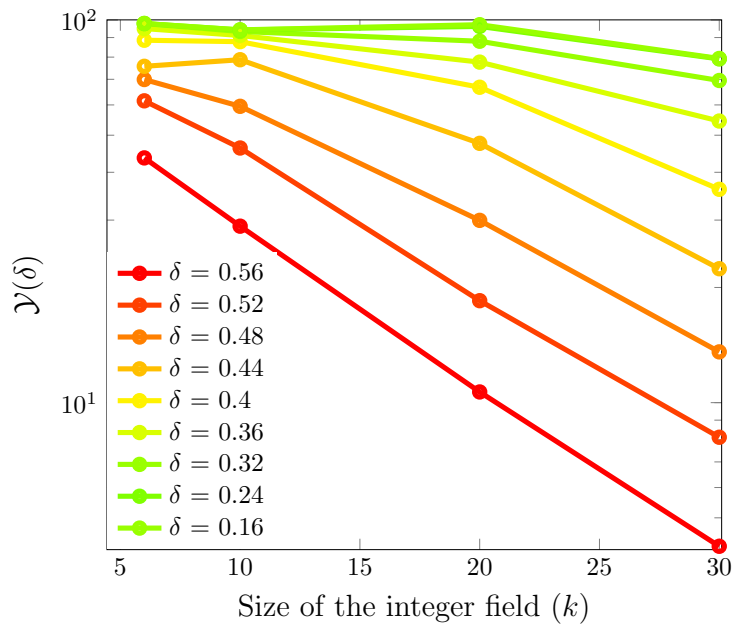
\begin{figure}[thbp]
\centering
\resizebox{0.6\textwidth}{!}{
%
%
\definecolor{mycolor1}{rgb}{1.00000,0.25098,0.00000}%
\definecolor{mycolor2}{rgb}{1.00000,0.50196,0.00000}%
\definecolor{mycolor3}{rgb}{1.00000,0.74902,0.00000}%
\definecolor{mycolor4}{rgb}{1.00000,0.94902,0.00000}%
\definecolor{mycolor5}{rgb}{0.85098,1.00000,0.00000}%
\definecolor{mycolor6}{rgb}{0.50196,1.00000,0.00000}%
\begin{tikzpicture}

\begin{axis}[%
scale only axis,
xmin=4.44081361618845,
xmax=30.1,
xtick={5,10,15,20,25,30},
label style = {font=\large},
xlabel={Size of the integer field ($k$)},
ymode=log,
ymin=4.12881255989588,
ymax=100.1,
yminorticks=true,
ylabel style={align=center},
ylabel={$\mathcal{Y(\delta)}$},
axis background/.style={fill=white},
legend style={at={(0,0)}, anchor=south west, legend cell align=left, align=left, draw=none}
]
\addplot [color=red, line width=2.0pt, mark=o, mark options={solid, red}]
  table[row sep=crcr]{%
6	43.62\\
10	28.92290249\\
20	10.66439909\\
30	4.213151927\\
};
\addlegendentry{$\delta\text{ = 0.56}$}

\addplot [color=mycolor1, line width=2.0pt, mark=o, mark options={solid, mycolor1}]
  table[row sep=crcr]{%
6	61.5\\
10	46.29931973\\
20	18.47165533\\
30	8.120181406\\
};
\addlegendentry{$\delta\text{ = 0.52}$}

\addplot [color=mycolor2, line width=2.0pt, mark=o, mark options={solid, mycolor2}]
  table[row sep=crcr]{%
6	69.92\\
10	59.50340136\\
20	29.95238095\\
30	13.57369615\\
};
\addlegendentry{$\delta\text{ = 0.48}$}

\addplot [color=mycolor3, line width=2.0pt, mark=o, mark options={solid, mycolor3}]
  table[row sep=crcr]{%
6	75.7\\
10	78.77324263\\
20	47.61451247\\
30	22.39002268\\
};
\addlegendentry{$\delta\text{ = 0.44}$}

\addplot [color=mycolor4, line width=2.0pt, mark=o, mark options={solid, mycolor4}]
  table[row sep=crcr]{%
6	88.56\\
10	87.79365079\\
20	66.73469388\\
30	36.11337868\\
};
\addlegendentry{$\delta\text{ = 0.4}$}

\addplot [color=mycolor5, line width=2.0pt, mark=o, mark options={solid, mycolor5}]
  table[row sep=crcr]{%
6	94.98\\
10	91.138322\\
20	77.63718821\\
30	54.45804989\\
};
\addlegendentry{$\delta\text{ = 0.36}$}

\addplot [color=green!20!lime, line width=2.0pt, mark=o, mark options={solid, green!20!lime}]
  table[row sep=crcr]{%
6	97.92\\
10	93.34693878\\
20	88.02040816\\
30	69.50793651\\
};
\addlegendentry{$\delta\text{ = 0.32}$}

\addplot [color=mycolor6, line width=2.0pt, mark=o, mark options={solid, mycolor6}]
  table[row sep=crcr]{%
6	97.92\\
10	94.2244898\\
20	96.23356009\\
30	79.138322\\
};
\addlegendentry{$\delta\text{ = 0.24}$}

\addplot [color=green!20!lime, line width=2.0pt, mark=o, mark options={solid, green!20!lime}]
  table[row sep=crcr]{%
6	97.92\\
10	94.2675737\\
20	97.24489796\\
30	79.36507937\\
};
\addlegendentry{$\delta\text{ = 0.16}$}

\end{axis}
\end{tikzpicture}%
}
\caption{Percentage of time spent by trajectories in the vicinity of assignments that satisfy at least $\delta$ fraction of equations versus the alphabet size $k$ for an instance with $\epsilon=0.4$. Note that the decay rate captures the hardness of the problem and the transition in scaling the assertion of the unique games conjecture.}
\label{fig:transition}
\end{figure}

\emph{Our numerically computed scaling is found to be consistent with the unique games conjecture as a function of $k$ and the gap between $1-\epsilon$ and $\delta$}.  In particular, we find that as the size of the integer field $k$ is increased, the curves transition from a polynomial/subexponential scaling (in green and yellow) to exponential scaling (in red). Recall that if the UGC is true, as $k$ increases, it should become harder to separate cases where the gap between $\delta$ and $1-\epsilon$ (the value gap) is larger. The less time trajectories spend satisfying a $\delta$ fraction of equations, the \emph{harder} the problem. 


Let the computational complexity of the UGC for different values of $(n_{\x},\delta,\epsilon,k)$ be captured by the scaling of the best possible classical algorithms for the problem, given by, $O(\text{exp}(n_{\x}^{f(\delta,\epsilon,k)}))$. If $f(\delta,\epsilon,k) < 1$, the algorithm is subexponential, and if $f(\delta,\epsilon,k)\geq 1$, it is exponential. We now set $\epsilon=0.4$ and use the data for $\mathcal{Y(\delta)}$, shown in Fig.~\ref{fig:transition} to estimate $f(\delta, k)$, shown in Fig.~\ref{fig:cval_transition}. One can see that $f(\delta,k)$ transitions from subexponential to expontential scaling values, \emph{consistent with statement of the UGC}. More details of the analysis are available in~\Cref{sec:numer}.
\begin{figure}
\centering
\resizebox{0.6\textwidth}{!}{
\input{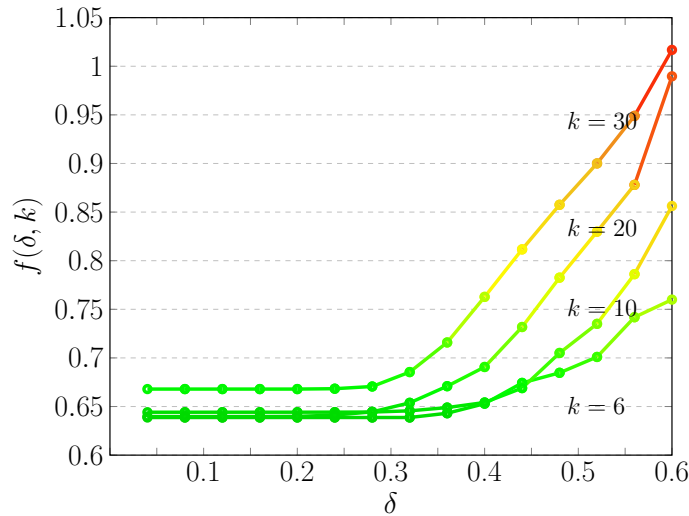}
}
\caption{Transition of the scaling exponent from Fig.~\ref{fig:transition} as a function of the size of the problem ($k$ values) and $\delta$. The colors depict the value of $f(\delta,k)$. Green and yellow colors correspond to subexponential values and red corresponds to exponential values ($f(\delta,k)\geq 1$). Results for $\epsilon=0.4$.}
\label{fig:cval_transition}
\end{figure}

One can define the UGC using completeness and soundness of the underlying $2$-Lin-$k$ instance~\cite{khot2002power}. Here, one must distinguish between the completeness case where there exists an assignment that satisfies at least a $c$ fraction of equations and the soundness case where all assignments satisfy less than an $s$ fraction of equations. The larger the gap between the $c$ and $s$ values, denoted by $\text{gap}(c,s)$, the easier the problem. Note that this definition is equivalent to the UGC statement used in our paper~\cite{trevisan2005approximation}. 

In~\cite{barak2018halfway}, Barak hypothesized a hardness picture based on the $(c,s)$ values where, as the gap decreases, the unique games transition from a subexponential to exponential regime (see~\Cref{sec:hardness}). We analyze the time spent by the trajectories in the vicinity of the optimal assignment as a function of $(1-\epsilon,\delta)$ (equivalent to the $(c,s)$ space) and find that the resulting plot (in Fig.~\ref{fig:bb}) reproduces the hypothesized image.  
A key consequence of our results is that \emph{if} these dynamical system embeddings of constraint satisfaction problems place fundamental limitations on the capabilities of deterministic Turing machines, then the \emph{unique games conjecture is indeed true}. Alternatively, if the UGC is proved to be true, then these \emph{dynamical systems embeddings capture fundamental limitations on algorithm construction for constraint satisfaction problems}.
\begin{figure}
\centering
\resizebox{0.6\textwidth}{!}{
\input{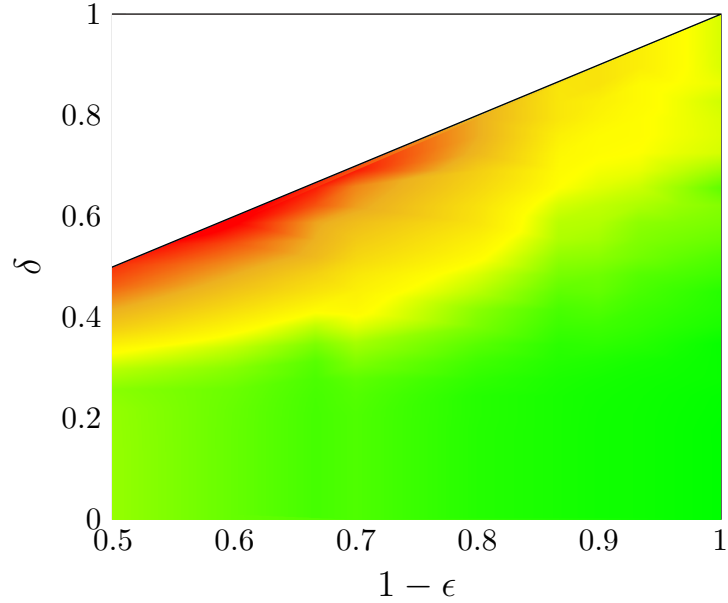}}
\caption{Transition of hardness as a function of $\delta$ and $\epsilon$. Color depicts the exponent. Red areas have exponential scaling. A similar picture was hypothesized in~\cite{barak2018halfway}. Results for $k=30$.}
\label{fig:bb}
\end{figure}

\section{Conclusions and future work}
The unique games conjecture has had exceptional impact on the field of computational complexity. By providing a novel means to characterize the hardness of computational problems, it has enabled the investigation of deep questions in complexity theory. In this work, we consider a prototypical definition of the UGC, namely, a linear system of two variables on $\Zk$. We then embed this linear system into a family of dynamical systems. It is shown that only valid solutions of the system of linear equations correspond to equilibria of the constructed dynamical system. Moreover, the family of dynamical systems does not admit any spurious attractors. We then prove that these dynamical systems are ergodic and numerically show that the weight of the invariant measure in the vicinity of various assignments depends on the value gap of the UGC and \emph{displays a transition that is consistent with the statement of the conjecture}.

In future work, we plan to provide theoretical bounds on the decay rates of the weights of the invariant measures. Moreover, given the deep connections between Turing machines~\cite{moore1990unpredictability} and the universality of dynamical systems~\cite{tao_2017}, we plan to study the universality of the dynamical systems in  Eqn.~\ref{eqn:dyn}. 
We conclude our work with three conjectures whose proofs will shed important light on  \emph{fundamental limits} on computation and its relationship with dynamical systems theory.
\begin{conjecture}
There always exist one-to-one maps from constraint satisfaction or optimization problems to dynamical systems such that the onset of hardness is captured by the underlying invariant measure. 
\end{conjecture}
\begin{conjecture}
For the UGC, as the alphabet size $k$ increases, the weight of the invariant measure in the vicinity of optimal or feasible assignment, transitions from subexponential to exponential scaling.    
\end{conjecture}
\begin{conjecture}
If the UGC is true, no sequence of operations of a Turing machine can circumvent the scaling properties of these dynamical systems across all problem instances.    
\end{conjecture}
Interesting extensions include the use of the dynamical systems to design efficient algorithms and compute boundaries of efficient algorithm construction for \emph{any} optimization or feasibility problem.

\section*{Acknowledgements}
This material is based upon work supported by the U.S. Navy / SPAWAR Systems Center Pacific under Contract No. N66001-18-C-4031. The authors thank M{\'a}ria Ercsey-Ravasz and Zolt{\'a}n Toroczkai for discussions related the dynamics of the system in the absence of solutions and escaping from arbitrary domains, Luca Trevisan for discussions related to the UGC, and John Guckenheimer for feedback on the dynamical system.

\clearpage
\appendix

\section{Mapping 2-Lin-$k$ instances to Disjunctive  and Conjunctive SAT Normal Forms}
\label{sec:mapping_conv}
As shown in the main text, each equation in the 2-Lin-$k$ system has the following form,
\begin{align}
\x_{i} =\x_{j} + \mathbf{b}_{l}\,\,\text{mod}\,\, k.
\label{aeqn:2lin}
\end{align}
Now, every variable $\x_{i}$, can be represented using spin variables (one spin value for each possible numeric assignment for $\x_{i}$), denoted by 
$S_{\x_{i}}=\{\s_{k(i-1)}, \s_{k(i-1) + 1},\s_{k(i-1) + 2},\hdots, \s_{k(i-1) + k-1}\} $, such that $\s_{l} \in \{-1,1\}$. Here, $\s_{l}=1$ and $\s_{l}=-1$ correspond to \texttt{true} and  \texttt{false} values respectively. Since $\x_{i}$ can only take one value, only one value in $S_{\x_{i}}$ must be $1$ or \texttt{true}. To ensure this property,  we impose additional constraints as shown in Eqn.~2 of the main text. It is easy to check that the mapping from $\x_{i}$ to $S_{\x_{i}}$ introduces $k$ spin variables and $M_{v} = 1 + \frac{k(k-1)}{2}$ clauses.

In the definiton of unique games, the 2-Lin-$k$ equations relating $\x_i$ and $\x_j$ are bijective mappings~\cite{trevisan2012khot} (as shown in Eqn.~\ref{aeqn:2lin}) . In other words, each potential integer assignment for $\x_j$, where $0\leq \x_j\leq k-1$, is mapped to a unique value for $\x_i$. This bijective property plays a critical role in the structure of the resulting Boolean form. It is easy to confirm that  Eqn.~\ref{aeqn:2lin} can be represented by the following Boolean formula,
\begin{align} \label{aeqn:dnf}
((\s_{k(i-1)}&\wedge \s_{k(j-1) + \bl\,\text{mod}\,k}) \vee (\s_{k(i-1) + 1}\wedge \s_{k(j-1) + \bl+1\,\text{mod}\,k}) \hdots\vee (\s_{ki-1}\wedge \s_{k(j-1) + \bl+k-1\,\text{mod}\,k})).
\end{align}
We note here that the assignment of $\s_{l} \in \{-1,1\}$ is equivalent to its corresponding Boolean assignment of $\{0,1\}$ or $\{\texttt{false},\texttt{true}\}$. Moreover, the $\text{mod } k$ operation acts only on the variable portion of the index (i.e., does not include the $k(j-1)$ term above). 

One can convert the above DNF formula (Eqn.~\ref{aeqn:dnf}) into a CNF form that preserves satisfiability~\cite{jackson2004optimality}, and avoids the exponential increase in the number of clauses. For every equation $q \in \{1, 2, \dots, n_{eq}\}$,  we introduce one additional variable $\z_q$. We then map each $\z_q$ to the spin variables given by, $S_{\z_q}=\{\s_{kn_x + k(q-1)}, \s_{kn_x + k(q-1) + 1}, \allowbreak \hdots, \s_{kn_x + k(q-1) + k-1}\}$, where $\s_l\in\{-1,1\}$. Then, it is easy to show that Eqn.~\ref{aeqn:dnf} is equivalent to the following conjunctive form,
\begin{align*} \label{clause}
&(\s_{kn_x + k(q-1)} \vee \s_{kn_x + k(q-1) + 1}, \vee \hdots \vee \s_{kn_x + k(q-1) + k-1}) \wedge\hdots \\
&(\neg \s_{kn_x + k(q-1)} \vee \s_{k(i-1)})  \wedge (\neg \s_{kn_x + k(q-1)} \vee  \s_{k(j-1) + \bl\,\text{mod}\,k}) \wedge\hdots \\
& (\neg \s_{kn_x + k(q-1) +1 } \vee \s_{k(i-1)+1})  \wedge (\neg \s_{kn_x + k(q-1)+1} \vee  \s_{k(j-1) + \bl + 1\, \text{mod}\,k}) \wedge \hdots\\
 & \vdots \\
 &(\neg \s_{kn_x + kq -1 } \vee \s_{ki-1})  \wedge (\neg \s_{kn_x + kq-1} \vee  \s_{k(j-1) + \bl+k-1\, \text{mod}\,k}). \numberthis
\end{align*}
It quickly follows that, for every equation, we add $N_q=k$ additional spin variables along with $M_q = 1+2k$ additional clauses.

Observe that the clause $(\s_{k(i-1)} \vee \s_{k(i-1) + 1}\hdots\vee \s_{ki -1})$, for every $\x_i$ in Eqn.~2 of the main text, is redundant, since one of variables in the clause \emph{must} be $1$ so as to satisfy the set of clauses given in Eqn.~\ref{clause}. Consequently, we ignore this term in our construction. Additionally, with the exception of one clause of size $k$, all other clauses are of size $2$. Thus, for $k=2$ the above mapping reduces to a 2-SAT instance~\cite{khot2002power}.

The number of resulting clauses, denoted by $M$, and variables, denoted by $N$, are given by,
\begin{align*}
M &= n_{\x}(M_v-1) + n_{\text{eq}}M_q, \\
&= n_{\x}\frac{k (k-1)}{2} + n_{\text{eq}}(1+2k),\\
N &= n_{\x} N_v + n_{\text{eq}}N_q, \\
&= (n_{\x}+n_{\text{eq}})k.
\end{align*}
One can then construct a mapping from the above Boolean formulae to a family of dynamical systems, such that, for \emph{any} 2-Lin-$k$ instance, the equilibria of the dynamical system correspond to solutions of the discrete system of equations with no spurious attractors~\cite{ercsey2011optimization}. We refer the reader to Eqn.~4 of the main text for further details. We now show that this dynamical system displays sensitive dependence to initial conditions. 

\section{Sensitive Dependence to Initial Conditions of Dynamical Systems resulting from \UNSAT~instances} \label{sec:sens_dep}
In this section, we prove that the dynamical systems (given by Eqn. 4 in the main text) that result from instances of unsatisfiable $K$-SAT problems or 2-Lin-$k$ equations, display sensitive dependence to initial conditions, a key property of chaotic systems~\cite{strogatz2018nonlinear}. We start by proving a simple lemma that will be important to prove the subsequent theorem.
\begin{lemma}\label{lem:set}
 Given an \UNSAT~instance $\mathcal{I}$ (of a $K$-SAT problem or 2-Lin-$k$ equations), there exists a non-empty subset of spin variables $\mathcal{S}$ such that, for all time $t$, the variables  $\mathcal{\s_{\mathcal{I}}}\in\mathcal{S}$ do not converge to $\{-1,1\}$.   
\end{lemma}
\begin{proof}
In~\cite{ercsey2011optimization}, the authors show that the resulting dynamical system only has equilibria at satisfiable assignments of the progenitor $K$-SAT. Note that solution clusters may form between two solutions that are only separated by a single spin flip. However, since we have restricted ourselves to \UNSAT~instances, $\mathcal{S}\neq\emptyset$. This is easy to prove by contradiction. Assume that the subset is empty, i.e. $\mathcal{S}=\emptyset$, then all spin variables have converged to values in $\{-1,1\}$. Since such an assignment would correspond to an equilibrium, it would contradict our \UNSAT~assumption for instance $\mathcal{I}$.
\end{proof}

\begin{theorem}\label{thm: lyapunov}
Any unsatisfiable (\UNSAT) instance of the $K$-SAT and 2-Lin-$k$ equations, that is used to define unique games, must display sensitive dependence to initial conditions (positive Lyapunov exponents). Specifically, unsatisfied clauses lead to exponential growth in values of the auxiliary variables that, in turn, cause local exponential growth of initial perturbations to the spin variables $\s_{l}$.
\end{theorem}
\begin{proof}
Let us show, analytically, that the sensitive dependence to initial conditions occurs as a consequence of the dynamics of the auxiliary variables, $a_m$. We use perturbation analysis to study the sensitivity of the variables to initial conditions. 

In the following, let $\mathbf{(s, a)}$ and $\mathbf{(s', a')}$ represent the original and perturbed variables of the system respectively. At $t=0$, let us assume that we perturb the initial conditions for a random spin variable $\s_r(0)$ that belongs to the set $\mathcal{S}$ (see lemma~\ref{lem:set} above), by an infinitesimal value $\delta$. All other variables in the dynamical system are left unperturbed,
\begin{align*}
\s_{r}^{'}(0) &= \s_{r}(0)+\delta,\\
\s_{p}^{'}(0) &= \s_{p}(0), \quad\quad \forall p \neq r \\
a_m'(0) &= a_m(0) \quad\quad \forall m =1,2,\cdots,M
\end{align*}
We note that the choice of $\s_{r}$ is completely arbitrary and one can pick any variable that belongs to the set $\mathcal{S}$. Moreover, one is free to perturb any number of variables from the set $\mathcal{S}$. Now, at any time $t$, $\mathbf{a}(t) = \mathbf{a}(0)e^{\int_{0}^{t}K_m^{\alpha}(\mathbf{s}(\tau))d\tau}$. Additionally, in the following analysis, for ease of notation, we will represent $K_m(\mathbf{s}(t))$ as $K_m(t)$ and, similarly, $K_{ml}(\mathbf{s}(t))$ will be $K_{ml}(t)$. The expressions for $K_m'(0)$ and $K_{ml}'(0)$ in terms of $K_m(0)$ and $K_{ml}(0)$ are obtained as follows,
\begin{align*}
K_m'(0) &= 2^{-k} \prod_{p=1}^{N} (1-c_{mp}\s_{p}^{'}(0)),\\
&= 2^{-k} (1-c_{mr}(\s_{r}(0)+\delta)) \prod_{\substack{p=1\\p\neq r}}^{N} (1-c_{mp}\s_{p}(0)),\\
&= 2^{-k} \prod_{p=1}^{N} (1-c_{mp}\s_{p}(0)) - 2^{-k} c_{mr}\delta \prod_{\substack{p=1\\p\neq r}}^{N} (1-c_{mp}\s_{p}(0)),\\
&= K_m(0) - \delta c_{mr}K_{mr}(0).
\end{align*}

Now, let $K_{ml}$ denote $K_m$ without the $(1-c_{ml}\s_{l})$ term. Then, for $l \neq r$,
\begin{align*}
K_{ml}'(0) &=2^{-k} \prod_{\substack{p=1\\p\neq l}}^{N} (1-c_{mp}\s_{p}^{'}(0)),\\
&= 2^{-k} (1-c_{mr}(\s_{r}(0)+\delta))\prod_{\substack{p=1\\p\neq l\\p\neq r}}^{N} (1-c_{mp}\s_p(0)),\\
&= K_{ml}(0) - \delta c_{mr}K_{mlr}(0).
\end{align*}

Where, $K_{mlr}$ simply denotes $K_m$ without the $(1-c_{ml}\s_{l})$ and $(1-c_{mr}\s_{r})$ terms,
\[K_{mlr} = 2^{-k} \prod_{\substack{p=1\\p\neq l\\p\neq r}}^{N} (1-c_{mp}\s_{p}).\] Note that for the case $l=r$, it is trivial to show that $K_{mr}'(0) = K_{mr}(0)$. To compute the sensitivity of the system of equations to the perturbation $\delta$, we calculate the derivatives of $\s_{l}^{'}$ at time $t=0$.
For $l \neq r$,
\begin{align*}
\frac{d\s_{l}^{'}}{dt}\Big|_{t=0}&= \sum_{m=1}^{M} 2a_m(0)c_{ml}K_{ml}'(0)K_m'(0),\\
&= \sum_{m=1}^{M} 2a_m(0)c_{ml}\Big( K_{ml}(0) - \delta c_{mr}K_{mlr}(0)\Big)\Big( K_m(0) - \delta c_{mr}K_{mr}(0)\Big),\\
&=  \sum_{m=1}^{M} 2a_m(0)c_{ml} K_{ml}(0) K_m(0) -\delta  \sum_{m=1}^{M} 2a_m(0)c_{ml}c_{mr} \Big(K_{mr}(0) K_{ml}(0)+K_{m}(0) K_{mlr}(0)\Big) \\& \qquad+ \delta^2 \sum_{m=1}^{M} 2a_m(0)c_{ml} K_{mr}(0) K_{mlr}(0),  \\
&=  \frac{d\s_{l}}{dt}\Big|_{t=0}-\delta  \sum_{m=1}^{M} 4a_m(0)c_{ml} c_{mr}K_{mr}(0) K_{ml}(0) + \delta^2 \sum_{m=1}^{M} 2a_m(0)c_{ml} K_{mr}(0) K_{mlr}(0).
\end{align*}

For the $l=r$ case we get,
\begin{align*}
\frac{d\s_{r}^{'}}{dt}\Big|_{t=0}&= \sum_{m=1}^{M} 2a_m(0)c_{mr}K_{mr}'(0)K_m'(0),\\
&= \sum_{m=1}^{M} 2a_m(0)c_{mr}K_{mr}(0)\Big( K_m(0) - \delta c_{mr}K_{mr}(0)\Big),\\
 &= \frac{d\s_{r}}{dt}\Big|_{t=0}- \delta  \sum_{m=1}^{M} 2a_m(0)K_{mr}^2(0).
\end{align*}

We now use the standard Euler approximation of the form, $\frac{d\s_{l}^{'}}{dt}\Big|_{t=0} \approx \frac{\s_{l}^{'}(\Delta t)-\s_{l}^{'}(0)}{\Delta t}$. For $l \neq r$ case we get,
\begin{align*}
\s_{l}^{'}(\Delta t) &= \s_{l}(\Delta t) + (\s_{l}^{'}(0)-\s_{l}(0))- \Delta t \delta  \sum_{m=1}^{M} 4a_m(0)c_{ml} c_{mr}K_{mr}(0) K_{ml}(0) \\&\qquad+ \Delta t \delta^2 \sum_{m=1}^{M} 2a_m(0)c_{ml} K_{mr}(0) K_{mlr}(0),\\
&= \s_{l}(\Delta t) - \Delta t \delta  \sum_{m=1}^{M} 4a_m(0)c_{ml} c_{mr}K_{mr}(0) K_{ml}(0) +\mathcal{O}(\delta^2 \Delta t),\\
&= \s_{l}(\Delta t) - \Delta t \delta  T_1^{r,l} +\mathcal{O}(\delta^2 \Delta t).
\end{align*}
Where we denote $T_1^{r,l}  = \sum_{m=1}^{M} 4a_m(0)c_{ml} c_{mr}K_{mr}(0) K_{ml}(0) $. For the $l=r$ case, the expression reduces to,
\begin{align*}
\s_{r}^{'}(\Delta t) &= \s_{r}(\Delta t) + (\s_{r}^{'}(0)-\s_{r}(0))- \Delta t  \delta  \sum_{m=1}^{M} 2a_m(0)K_{mr}^2(0),\\
&= \s_{r}(\Delta t) + \delta -\Delta t  \delta  \sum_{m=1}^{M} 2a_m(0)K_{mr}^2(0),\\
&=  \s_{r}(\Delta t) + \delta -\Delta t  \delta  T_2^{r,r}.
\end{align*}
Here, we set $T_2^{r,r} = \sum_{m=1}^{M} 2a_m(0)K_{mr}^2(0)$. \\

The above expressions capture the effect of the perturbation $\delta$ on the spin variables after a single time step $\Delta t$. We now evolve the perturbation forward in time to assess its impact on the long-term dynamics of the system. \\

It can be shown that at time $t$,
$$K_m'(t) = K_m(t) - \delta c_{mr}K_{mr}(t) + \mathcal{O}(\delta \Delta t),$$
$$K_{ml}'(t) =K_{ml}(t) - \delta c_{mr}K_{mlr}(t)+ \mathcal{O}(\delta \Delta t),$$
$$K_{mr}'(t) =K_{mr}(t) + \mathcal{O}(\delta \Delta t).$$
Note that, in the above expressions, we neglect second order and above terms (in $\delta$), since they will result in third order terms in the final equations (which are eventually neglected). This gives,

\begin{align*}
 \s_{l}^{'}(t+\Delta t) &= \s_{l}(t+\Delta t)-\Delta t \delta  \sum_{m=1}^{M}\sum_{n=0}^{t/\Delta t} 4a_m'(n\Delta t)c_{ml}c_{mr}K_{mr}(n\Delta t) K_{ml}(n\Delta t) +\mathcal{O}(\delta^2 \Delta t),
 \end{align*}
 \begin{align*}
 \s_{r}^{'}(t+\Delta t) &= \s_{r}(t+\Delta t) + \delta  - \Delta t \delta  \sum_{m=1}^{M}\sum_{n=0}^{t/\Delta t} 2a_m'(n\Delta t)K_{mr}^2(n\Delta t).
 \end{align*}
 Here, $$a_m'(n\Delta t) = a_m(n\Delta t)e^{-\Delta t\delta c_{mr}\sum_{j=0}^{n-1}K_{mr}(j\Delta t)}.$$


We use the standard definition of the Lyapunov exponent~\cite{strogatz2018nonlinear} that captures the rate of separation of trajectories that start infinitesimally close to one another. It is a standard tool used to quantify the ``sensitive dependence to initial conditions'' requirement of chaotic dynamics,
\[\lambda_r = \lim_{t \to \infty} \lim_{\Delta \s_{r}(0)\to 0} \frac{1}{t} \ln \frac{|\Delta \s_{r}(t)|}{|\Delta \s_{r}(0)|}.\]
Now, if $\Delta\s_{r}(0) = \delta$, is the initial separation between trajectories, the separation between $\s_{r}^{'}(t)$ and $\s_{r}(t)$ can be approximated as follows,
\begin{align*}
|\Delta\s_{r}(t)| &= |\s_{r}^{'}(t)-\s_{r}(t)|,\\
&= |\delta  - \Delta t \delta  \sum_{m=1}^{M}\sum_{n=0}^{t/\Delta t} 2a_m'(n\Delta t)K_{mr}^2(n\Delta t)|,\\
& \approx |\Delta t \delta  \sum_{m=1}^{M}\sum_{n=0}^{t/\Delta t} 2a_m'(n\Delta t)K_{mr}^2(n\Delta t)|.
\end{align*}
Note that the above approximation is true since the magnitude of the second term dominates $\delta$ as $t \to \infty$. This results in,
\begin{align*}
\lim_{\delta\to 0} \ln \frac{|\Delta \s_{r}(t)|}{|\delta|} &= \lim_{\delta\to 0}  \ln\Bigg(\Big|\Delta t  \sum_{m=1}^{M}\sum_{n=0}^{t/\Delta t} 2a_m'(n\Delta t)K_{mr}^2(n\Delta t)\Big| \Bigg),\\
&= \lim_{\delta\to 0}  \ln\Bigg(\Big|\Delta t  \sum_{m=1}^{M}\sum_{n=0}^{t/\Delta t} 2a_m(n\Delta t)e^{-\Delta t\delta c_{mr}\sum_{j=0}^{n-1}K_{mr}(j\Delta t)}K_{mr}^2(n\Delta t)\Big| \Bigg),\\
&=  \ln\Bigg(\Big|\Delta t  \sum_{m=1}^{M}\sum_{n=0}^{t/\Delta t} 2a_m(n\Delta t)K_{mr}^2(n\Delta t)\Big| \Bigg),\\
&\geq  \ln\big(2 a_g(t)K_{gr}^2(t) \Delta t\big).
\end{align*}
where $g$ is the index of any clause that contains $\x_r$. Here, the choice of the clause and, consequently, the value of $g$ is left to the reader. Since each term in the double summation is non-negative, this choice has no impact on the final result. Therefore, we get,
\begin{align*}
\lambda_r &= \lim_{t \to \infty} \lim_{\delta\to 0} \frac{1}{t} \ln \frac{|\Delta\s_{r}(t)|}{|\delta|},\\
& \geq \lim_{t \to \infty} \frac{1}{t} \ln\big(2 a_g(t)K_{gk}^2(t) \Delta t\big),\\
&= \lim_{t \to \infty} \frac{1}{t} \ln\big(2 a_g(0)e^{\int_{0}^{t}K_g^{\alpha}(\tau)d\tau}K_{gk}^2(t) \Delta t\big),\\
&= \lim_{t \to \infty} \frac{1}{t} \int_{0}^{t}K_g^{\alpha}(\tau)d\tau + \frac{1}{t} \ln\big(2 a_g(0)K_{gk}^2(t) \Delta t\big).\\
\end{align*}
Since the limit of the second term goes to zero as $t\to \infty$ we get the following expression,
\begin{align}
\lambda_r = \lim_{t \to \infty} \frac{1}{t} \int_{0}^{t}K_g^{\alpha}(\tau)d\tau.
\end{align}
It is easy to check that, for any \UNSAT~problem,  time $t$, and $\epsilon>0$, there exists a time interval $t'(\alpha, \epsilon)$, such that for some $\epsilon > 0$, \[\int_{t}^{t+t'}K_g ^{\alpha}(\tau)d\tau \geq \epsilon. \]
Since we will be taking the limit of $t\rightarrow\infty$, we assume that $t>>t'(\alpha, \epsilon)$. By taking $\frac{t}{t'(\alpha, \epsilon)}$ intervals we can show that,
\begin{align}
 \lambda_r(\alpha) &\geq \lim_{t \to \infty} \frac{1}{t} \frac{t}{t'(\alpha, \epsilon)}\epsilon,\nonumber\\
 &\geq  \frac{\epsilon}{t'(\alpha, \epsilon)}.
\end{align}
Since $\epsilon>0$, $\lambda_r$ \emph{must} be positive. Consequently, the maximum Lyapunov exponent (MLE) must also be positive, given by $\displaystyle\max_{r}\lambda_r > 0$. A positive MLE is considered as a key indicator of chaos. Moreover, we observe that the chaos in the system originates from the natural exponential growth dynamics of the auxiliary variables $a_m$. We remind the reader that, although the $a_m$ dynamics are unbounded, the dynamics of the spin variables $\s$ remain bounded to the cube $[-1,1]^N$.
\end{proof}
\begin{remark}
\label{remark:fsle}
Note that since $K_g(\tau) \in [0,1]$, as $\alpha $ increases, $(K_g(\tau))^{\alpha}$ decreases.  Consequently, the time interval $t'(\alpha, \epsilon)$ required for $\int_{t}^{t+t'}K_g^{\alpha}(\tau)d\tau \geq \epsilon$ increases. This, in turn, causes $\lambda_k(\alpha)$ to decrease as $\alpha$ increases (see Fig.~\ref{fig:fsle_average} for numerical confirmation).
\end{remark}
\begin{figure}
\centering
\resizebox{0.6\textwidth}{!}{
%
%
\begin{tikzpicture}

\begin{axis}[%
xmin=1,
xmax=4,
xlabel={$\alpha$},
ymode=log,
ymin=0.000932,
ymax=0.1019,
yminorticks=true,
ylabel={Average FSLE},
]
\addplot [color=blue, line width=2.0pt, mark=square, mark options={solid, blue}, forget plot]
  table[row sep=crcr]{%
1	0.1019\\
2	0.0456\\
3	0.0128\\
4	0.000932\\
};
\end{axis}

\end{tikzpicture}%
}
\caption{The average finite size Lyapunov exponent as a function of the exponent $\alpha$ for randomly generated instances of the UGC. The simulations were averaged for two different instances with $50$ randomly generated initial conditions in each instance.}
\label{fig:fsle_average}
\end{figure}
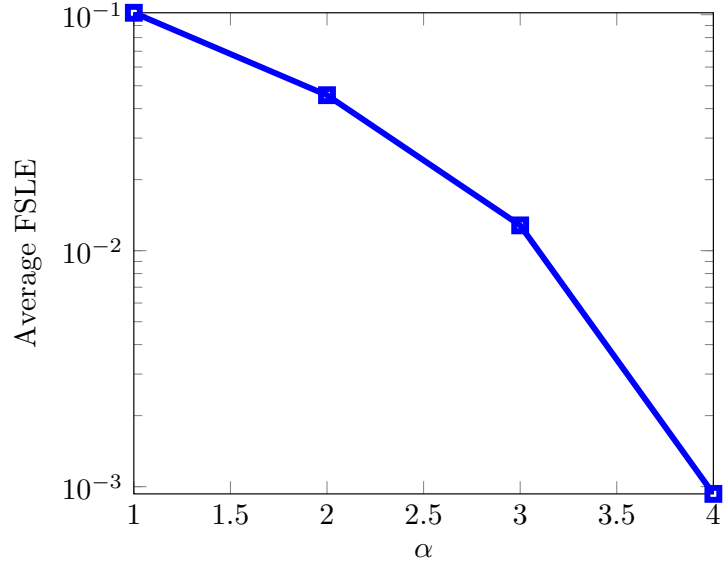

\begin{remark}
Interestingly, the initial distribution of the auxiliary variables $a_m$ has a significant influence on the onset of chaos in the system as shown in Figure \ref{FSLE}. The figure compares the finite size Lyapunov exponent (FSLE) for two different initializations of $a_m$ for an \UNSAT~system with clause density close to the phase transition or frozen regime~\cite{ercsey2011optimization}. This result emphasizes the fact that the auxiliary variables strongly influence the sensitive dependence to initial conditions in the underlying dynamical system.

\begin{figure}[htbp]
        \centering
        \begin{subfigure}[b]{0.49\textwidth}
            \centering
             \includegraphics[width=\textwidth]{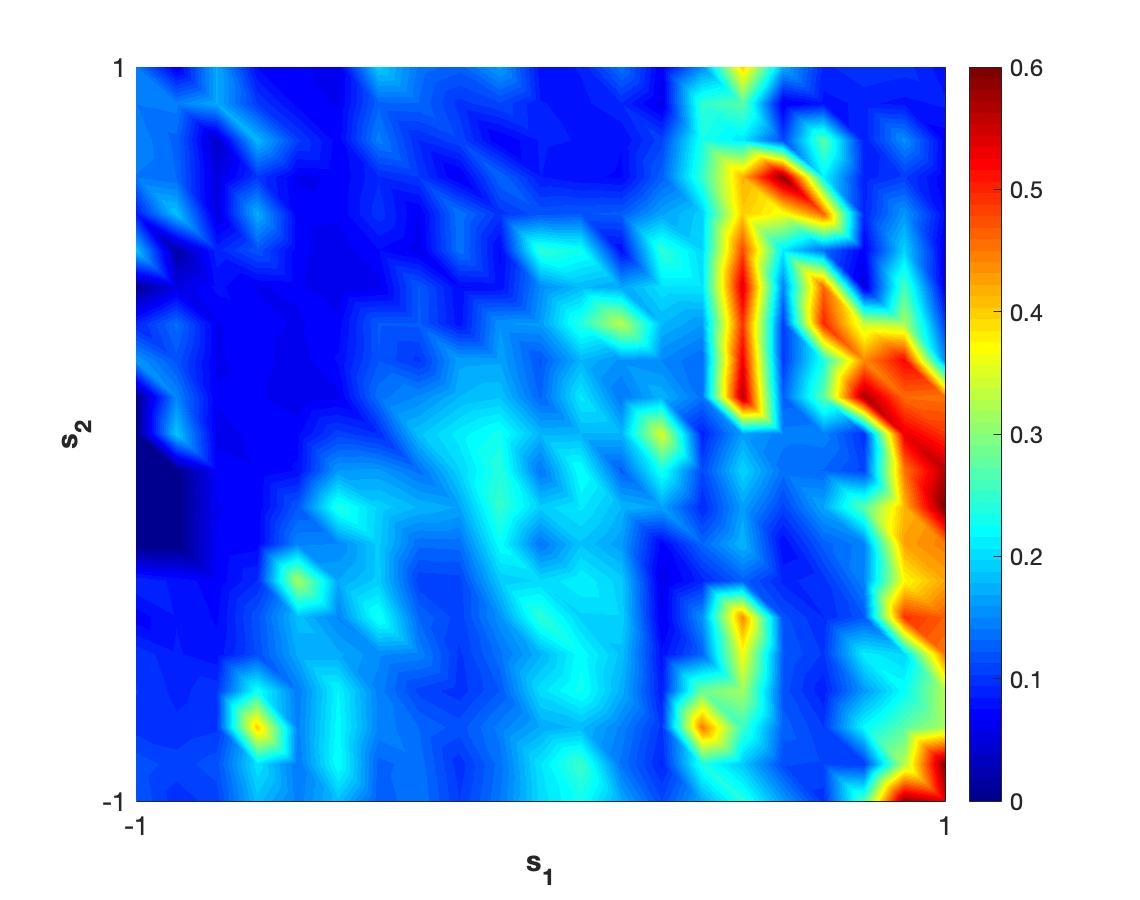}
            \caption{ $a_{m}$ initialized to 1}
            \end{subfigure}
            \hfill
        \begin{subfigure}[b]{0.49\textwidth}
            \centering
             \includegraphics[width=\textwidth]{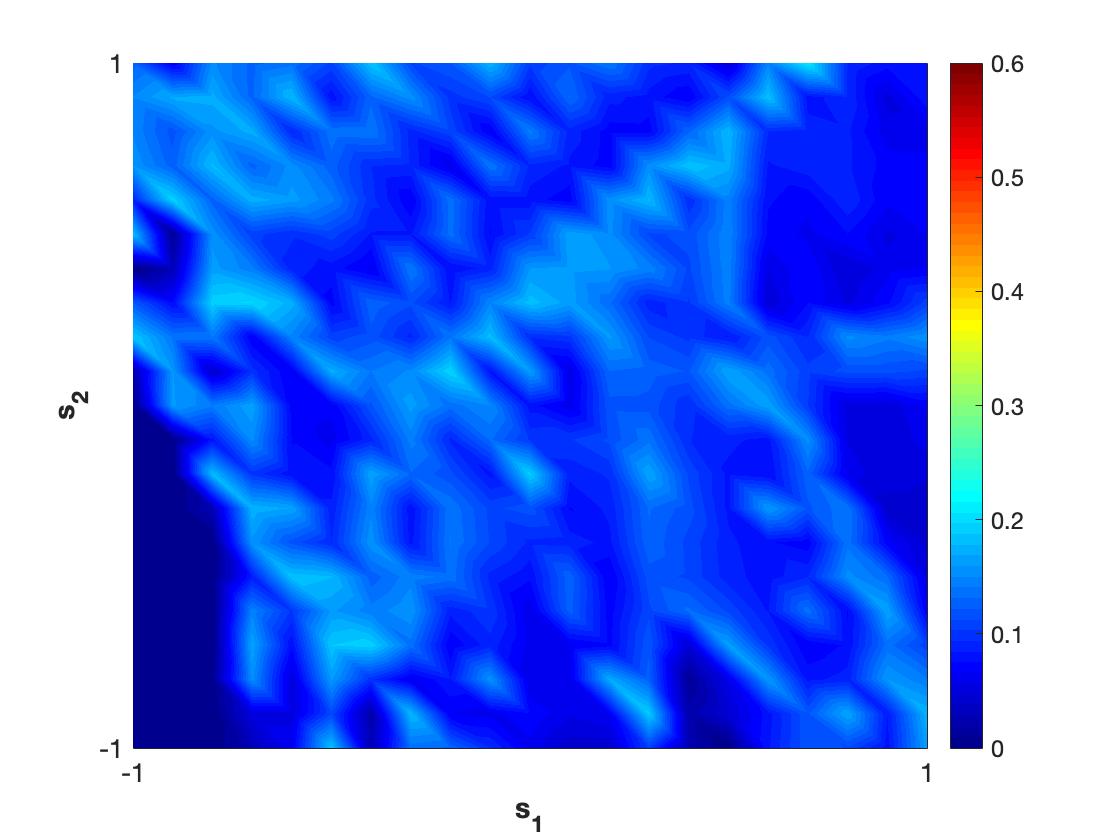}
            \caption{$a_{m}$ chosen uniformly at random between $(0,1)$}
        \end{subfigure}
        \caption{Comparing the finite size Lypanov exponent (FSLE) for an \UNSAT~system. The FSLE values are represented as colors. Note the prominence of chaotic behaviour in the case when the auxiliary variables are \emph{all} initialized to 1.}

            \label{FSLE}
\end{figure}
\end{remark}

\begin{remark}
As mentioned previously, in satisfiable (\texttt{sat}) instances, the convergence of the dynamical system to a global minima strongly depends on the dynamics of the auxiliary variables. The inherent exponential growth of $a_m$ prevents the existence of stable limit cycles and spurious attractors~\cite{ercsey2011optimization}. Additionally, it ensures that stable fixed points correspond to solutions of the SAT problem. However, the exponential growth of $a_m$ results in highly stiff systems that increase computational cost. In this work, we use the following dynamics for the auxiliary variables ($\alpha=2$), 
\[\frac{da_m}{dt} = a_mK_m^2, \hspace{1cm} \forall m=1,2,\hdots,M.\]
As shown before, $K_m \in [-1,1]$, consequently, $K_m^2 \leq K_m$. While the modified dynamics continues to ensure the convergence of the solution of the SAT problem without getting trapped in local minima, the resulting dynamical system is less stiff and results in more desirable convergence properties for numerical integration.
\end{remark}

\section{Numerical results for the dynamical system}
\label{sec:numer}
In this section, we provide additional details on computing the scaling of the exponent $f(\delta,\epsilon,k)$. As mentioned in the main text, the computational complexity of the UGC for different values of $(n_{\x},\delta,\epsilon,k)$ can be captured by the scaling of the best possible algorithms for the problem, given by, $O(\text{exp}(n_{\x}^{f(\delta,\epsilon,k)}))$. One can check that if $f(\delta,\epsilon,k) < 1$, the algorithm is subexponential, while $f(\delta,\epsilon,k)\geq 1$ implies exponential scaling.

Let $\mathcal{Y}$ be the fraction of time that the trajectory spends in the vicinity of any assignment that satisfies at least a $\delta$ fraction of the equations. We can then capture the scaling of $\mathcal{Y}$ as follows,
\begin{align}
\label{eq:fscale}
 \mathcal{Y} &= O(\exp{(-n_{\x}^{f(\delta,\epsilon,k)})}), \nonumber\\  
   &= \beta\exp{(-n_{\x}^{f(\delta,\epsilon,k)})}, \nonumber \\
  \ln(\mathcal{Y}) &= \ln\beta - n_{\x}^{f(\delta,\epsilon,k)}, \nonumber \\
  f(\delta,\epsilon,k) &= \log_{n_{\x}}(\ln\beta - \ln(\mathcal{Y})).
\end{align}
For each value of $\epsilon$ and $k$, we generate a random instance of the unique games as described in Fig.~3 of the main text. We compute $\mathcal{Y(\delta)}$  by simulating the dynamical system (integrating Eqn.~4 of the main text), and estimating the decay rates of weight of the invariant measure in an $L_1$ ball around assignments that solve at least a $\delta$ fraction of equations. 

The corresponding dynamical system is simulated for $600$ sec in the Matlab programming language. Numerical integration is performed using ODE solverssuch as \texttt{ode45}, and \texttt{ode15s} or \texttt{ode23tb} for stiffer cases. The time spent by the system in the vicinity of assignments satisfying $0 < \delta \leq 1-\epsilon$ fraction of equations is computed from the simulated trajectories. Specifically, $\mathcal{Y}(\delta)$ is computed as the fraction of the time spent satisfying at least a $\delta$ fraction of the equations. Note that transient dynamics, when switching from one variable assignment (in the original 2-Lin-$k$ system) to another, are ignored. In other words, the transient dynamics time is spent satisfying no equations, and is, therefore, not included in the computation of $\mathcal{Y}(\delta)$. Our results are compiled by averaging over $300-500$ random initializations of the spins and auxiliary variables. Our numerical experiments were performed for problems of the size $n_{\x} \approx 11$ and $n_{\text{eq}} \approx 30$. 

We use the expression in Eqn.~\ref{eq:fscale} to compute $f(\delta, \epsilon, k)$ from $\mathcal{Y}$. Specifically, for each data point in $\mathcal{Y}(\delta)$, we compute the corresponding $f(\delta, \epsilon, k)$. For example, see Fig.~\ref{fig:exponent_k30}, where we present $f(\delta,\epsilon)$ for $k=30$. Here, one can clearly observe the transition of the problem from subexponential scaling to the exponential scaling consistent with the onset of NP-hardness. In the main text, Fig.~6 presents a similar plot where $\epsilon$ is held constant.
\begin{figure}
    \centering
    \includegraphics[width=0.9\textwidth]{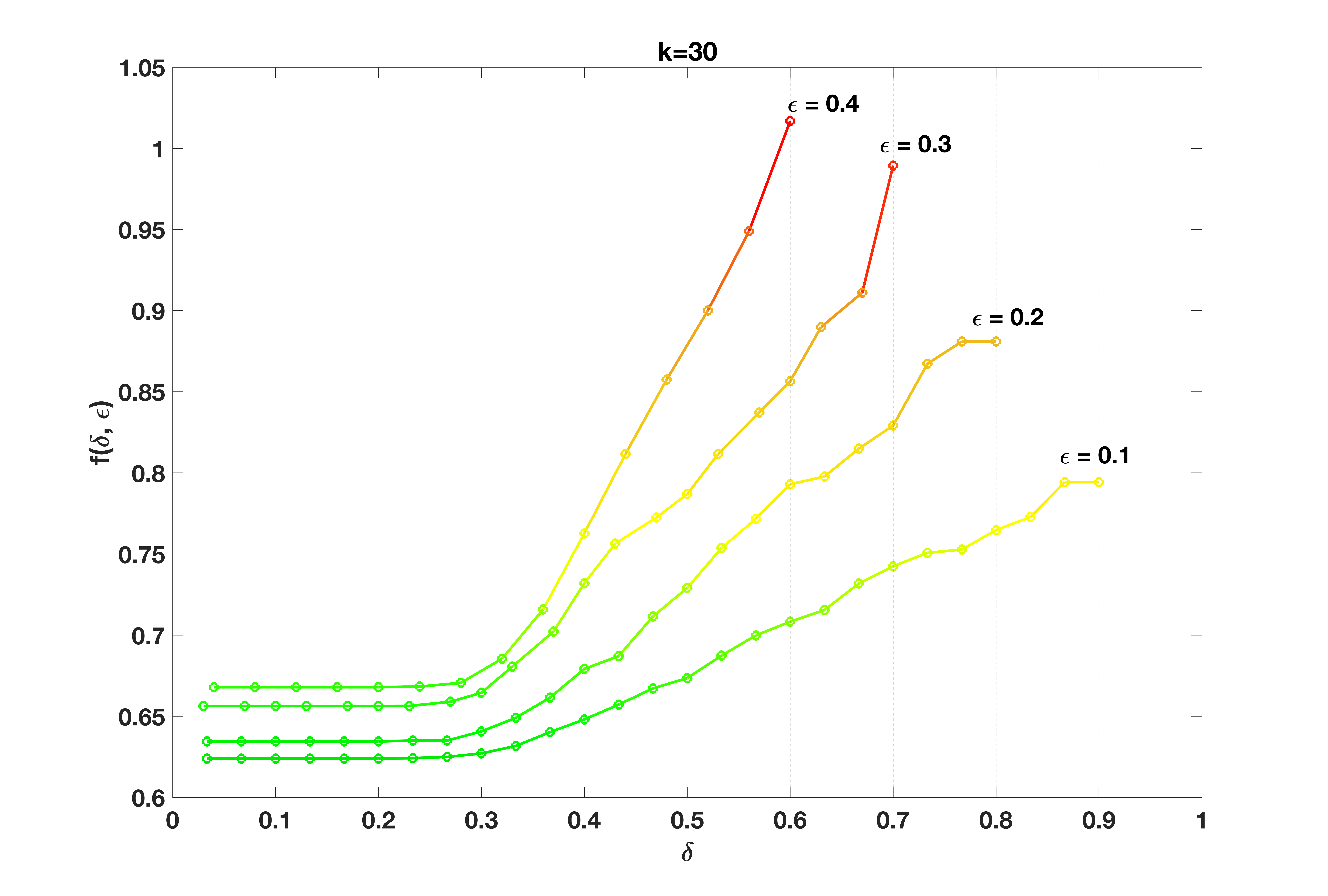}
    \caption{Exponent that captures the scaling for solving $2$-Lin-$k$ systems of equations on a finite field of $k=30$ as a function of $\delta$.}
    \label{fig:exponent_k30}
\end{figure}

\section{A Hardness Landscape Hypothesis}
\label{sec:hardness}
In~\cite{barak2018halfway}, Boaz Barak hypothesized a hardness landscape in terms of the completeness $(c)$ and soundness $(s)$ formulation of the unique games conjecture~\cite{khot2002power}. We note that the $(c,s)$ formulation is equivalent to the $(\epsilon,\delta)$ used in our work. In the $(c,s)$ formulation, given a $2$-Lin-$k$ system, the goal of the unique games conjecture, is to distinguish between the completeness and soundness cases. In the completeness case, there exists an assignment such that a $c$ fraction of equations are satisfied. While in the soundness case, \emph{all} assignments satisfy at most an $s$ fraction of equations. 

It is easy to check that the problem becomes easier as the gap between $c$ and $s$ increases. In~\cite{barak2018halfway}, Barak hypothesized that the hardness landscape has a structure as depicted in Fig.~\ref{fig:hardness_landscape}. In our work, for fixed $k$, we computed the fraction of time the trajectory spends in the vicinity of the optimal assignment and found that it reproduces Barak's hypothesized hardness landscape. The computational results display a clean transition from subexponential to exponential scaling. See Fig.~6 in the main text and associated discussions. These results seem to point to the fact that the UGC is likely true and there exist deep connections between computational complexity and dynamical systems theories. 

\begin{figure}
    \centering
    \includegraphics[width=0.8\textwidth]{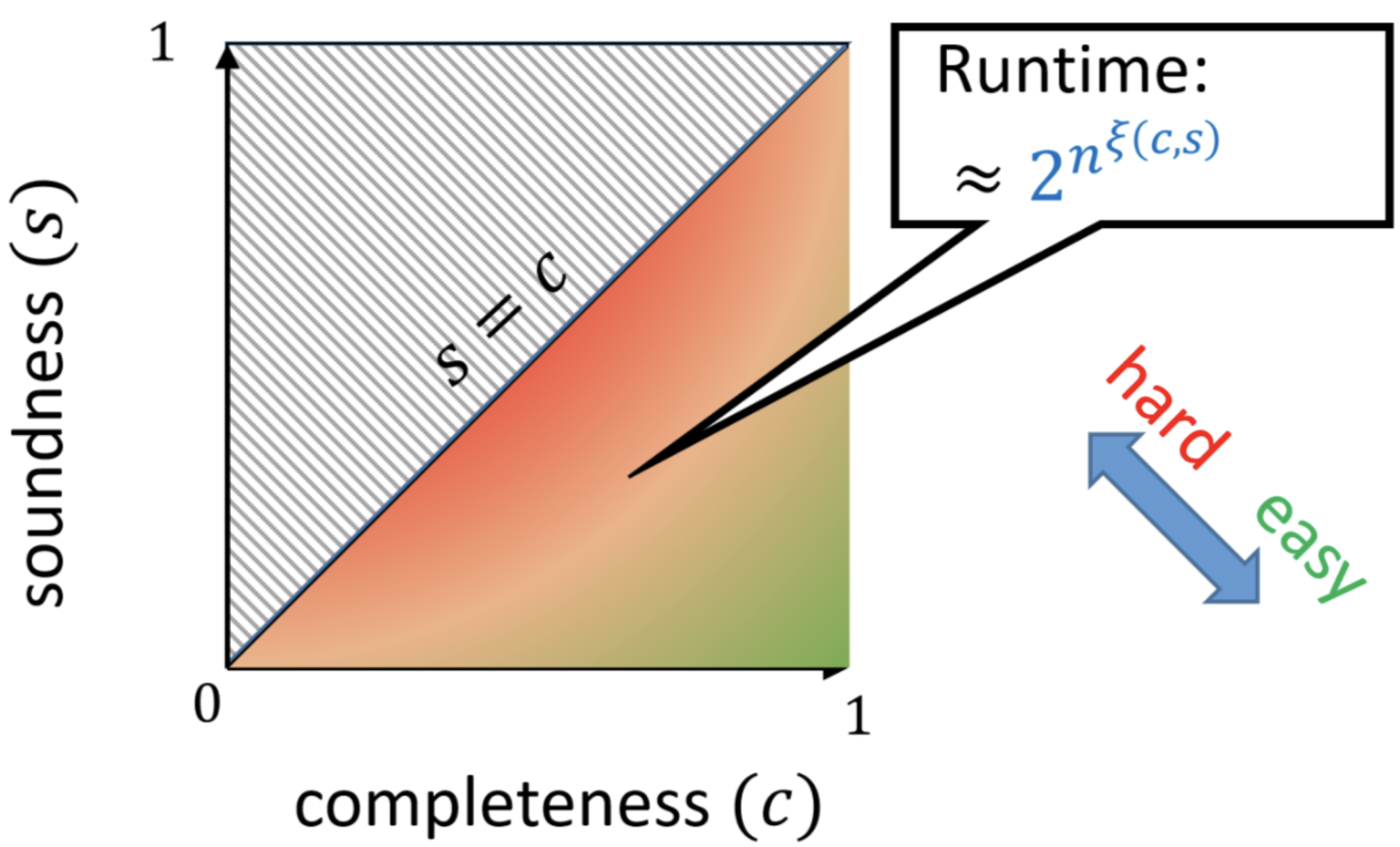}
    \caption{Hardness landscape hypothesized in~\cite{barak2018halfway}}.
    \label{fig:hardness_landscape}
\end{figure}

\bibliographystyle{unsrt}
\bibliography{ugc_bib}

\begin{thebibliography}{10}

\bibitem{arora2010subexponential}
Sanjeev Arora, Boaz Barak, and David Steurer.
\newblock Subexponential algorithms for unique games and related problems.
\newblock In {\em Foundations of Computer Science (FOCS), 2010 51st Annual IEEE Symposium on}, pages 563--572. IEEE, 2010.

\bibitem{FSLE}
Erik Aurell, Guido Boffetta, Andrea Crisanti, Giovanni Paladin, and Angelo Vulpiani.
\newblock Growth of noninfinitesimal perturbations in turbulence.
\newblock {\em Physical review letters}, 77(7):1262, 1996.

\bibitem{barak2018halfway}
Boaz Barak.
\newblock Unique games conjecture – halfway there?, 2018.

\bibitem{Cit:belkin}
Mikhail Belkin and Partha Niyogi.
\newblock Towards a theoretical foundation for {L}aplacian-based manifold methods.
\newblock In {\em International Conference on Computational Learning Theory}, pages 486--500. Springer, 2005.

\bibitem{Bro89}
R.~W. Brockett.
\newblock Least squares matching problems.
\newblock {\em Linear Algebra and its Applications}, 122--124:761--777, 1989.

\bibitem{Bro91}
R.~W. Brockett.
\newblock Dynamical systems that sort lists, diagonalize matrices and solve linear programming problems.
\newblock {\em Linear Algebra and Its Applications}, 146:79--91, 1991.

\bibitem{charikar2009every}
Moses Charikar, Venkatesan Guruswami, and Rajsekar Manokaran.
\newblock Every permutation {CSP} of arity 3 is approximation resistant.
\newblock In {\em 2009 24th Annual IEEE Conference on Computational Complexity}, pages 62--73. IEEE, 2009.

\bibitem{charikar2006near}
Moses Charikar, Konstantin Makarychev, and Yury Makarychev.
\newblock Near-optimal algorithms for unique games.
\newblock In {\em Proceedings of the thirty-eighth annual ACM symposium on Theory of computing}, pages 205--214. ACM, 2006.

\bibitem{cook1971complexity}
Stephen~A Cook.
\newblock The complexity of theorem-proving procedures.
\newblock In {\em Proceedings of the third annual ACM symposium on Theory of computing}, pages 151--158. ACM, 1971.

\bibitem{cornfeld2012ergodic}
Isaac~P Cornfeld, Sergej~V Fomin, and Yakov~Grigor'evǐc Sinai.
\newblock {\em Ergodic theory}, volume 245.
\newblock Springer Science \& Business Media, 2012.

\bibitem{ercsey2011optimization}
M{\'a}ria Ercsey-Ravasz and Zolt{\'a}n Toroczkai.
\newblock Optimization hardness as transient chaos in an analog approach to constraint satisfaction.
\newblock {\em Nature Physics}, 7(12):966, 2011.

\bibitem{ercsey2012chaos}
M{\'a}ria Ercsey-Ravasz and Zolt{\'a}n Toroczkai.
\newblock The chaos within {S}udoku.
\newblock {\em Scientific reports}, 2:725, 2012.

\bibitem{feige2007understanding}
Uriel Feige, Guy Kindler, and Ryan O'Donnell.
\newblock Understanding parallel repetition requires understanding foams.
\newblock In {\em Computational Complexity, 2007. CCC'07. Twenty-Second Annual IEEE Conference on}, pages 179--192. IEEE, 2007.

\bibitem{frigg2011ergodic}
Roman Frigg, Joseph Berkovitz, and Fred Kronz.
\newblock The ergodic hierarchy.
\newblock 2011.

\bibitem{goldreich2006promise}
Oded Goldreich.
\newblock On promise problems: A survey.
\newblock In {\em Theoretical computer science}, pages 254--290. Springer, 2006.

\bibitem{guckenheimer1983local}
John Guckenheimer and Philip Holmes.
\newblock Local bifurcations.
\newblock In {\em Nonlinear oscillations, dynamical systems, and bifurcations of vector fields}, pages 117--165. Springer, 1983.

\bibitem{isaksson2012maximally}
Marcus Isaksson and Elchanan Mossel.
\newblock Maximally stable {G}aussian partitions with discrete applications.
\newblock {\em Israel Journal of Mathematics}, 189(1):347--396, 2012.

\bibitem{jackson2004optimality}
Paul Jackson and Daniel Sheridan.
\newblock The optimality of a fast {CNF} conversion and its use with {SAT}.
\newblock {\em Proc. of SAT}, 2004.

\bibitem{katok1995introduction}
Anatole Katok and Boris Hasselblatt.
\newblock {\em Introduction to the modern theory of dynamical systems}, volume~54.
\newblock Cambridge university press, 1995.

\bibitem{khot2002power}
Subhash Khot.
\newblock On the power of unique 2-prover 1-round games.
\newblock In {\em Proceedings of the thiry-fourth annual ACM symposium on Theory of Computing}, pages 767--775. ACM, 2002.

\bibitem{khot2007optimal}
Subhash Khot, Guy Kindler, Elchanan Mossel, and Ryan O’Donnell.
\newblock Optimal inapproximability results for {MAX-CUT} and other 2-variable {CSP}s?
\newblock {\em SIAM Journal on Computing}, 37(1):319--357, 2007.

\bibitem{khot2016candidate}
Subhash Khot and Dana Moshkovitz.
\newblock Candidate hard unique game.
\newblock In {\em Proceedings of the forty-eighth annual ACM symposium on Theory of Computing}, pages 63--76. ACM, 2016.

\bibitem{khot2008vertex}
Subhash Khot and Oded Regev.
\newblock Vertex cover might be hard to approximate to within 2- $\varepsilon$.
\newblock {\em Journal of Computer and System Sciences}, 74(3):335--349, 2008.

\bibitem{khot2005unique}
Subhash Khot and Nisheeth~K Vishnoi.
\newblock On the unique games conjecture.
\newblock In {\em FOCS}, volume~5, page~3, 2005.

\bibitem{monasson1999determining}
R{\'e}mi Monasson, Riccardo Zecchina, Scott Kirkpatrick, Bart Selman, and Lidror Troyansky.
\newblock Determining computational complexity from characteristic ‘phase transitions’.
\newblock {\em Nature}, 400(6740):133, 1999.

\bibitem{moore1990unpredictability}
Cristopher Moore.
\newblock Unpredictability and undecidability in dynamical systems.
\newblock {\em Physical Review Letters}, 64(20):2354, 1990.

\bibitem{o2012new}
Ryan O'Donnell and John Wright.
\newblock A new point of {NP}-hardness for unique games.
\newblock In {\em Proceedings of the forty-fourth annual ACM symposium on Theory of Computing}, pages 289--306. ACM, 2012.

\bibitem{Cit:sahai_hearing}
T.~Sahai, Alberto Speranzon, and Andrzej Banaszuk.
\newblock Hearing the clusters of a graph: A distributed algorithm.
\newblock {\em Automatica}, 48(1):15--24, 2012.

\bibitem{sahai2020dynamical}
Tuhin Sahai.
\newblock Dynamical systems theory and algorithms for np-hard problems.
\newblock {\em Advances in Dynamics, Optimization and Computation}, pages 183--206, 2020.

\bibitem{sahai2010wave}
Tuhin Sahai, Alberto Speranzon, and Andrzej Banaszuk.
\newblock Wave equation based algorithm for distributed eigenvector computation.
\newblock In {\em Decision and Control (CDC), 2010 49th IEEE Conference on}, pages 7308--7315. IEEE, 2010.

\bibitem{Cit:TSP}
Tuhin Sahai, Adrian Ziessler, Stefan Klus, and Michael Dellnitz.
\newblock Continuous relaxations for the traveling salesman problem.
\newblock {\em Nonlinear Dynamics}, 97(4):2003--2022, 2019.

\bibitem{strogatz2018nonlinear}
Steven~H Strogatz.
\newblock {\em Nonlinear Dynamics and Chaos with Student Solutions Manual: With Applications to Physics, Biology, Chemistry, and Engineering}.
\newblock CRC Press, 2018.

\bibitem{Cit:Nesterov_dyn}
Weijie Su, Stephen Boyd, and Emmanuel Candes.
\newblock A differential equation for modeling {N}esterov’s accelerated gradient method: Theory and insights.
\newblock In {\em Advances in Neural Information Processing Systems}, pages 2510--2518, 2014.

\bibitem{tao_2017}
Terence Tao.
\newblock On the universality of potential well dynamics.
\newblock {\em Dynamics of Partial Differential Equations}, 14(3):219--238, 2017.

\bibitem{trevisan2005approximation}
Luca Trevisan.
\newblock Approximation algorithms for unique games.
\newblock In {\em Foundations of Computer Science, 2005. FOCS 2005. 46th Annual IEEE Symposium on}, pages 197--205. IEEE, 2005.

\bibitem{trevisan2012khot}
Luca Trevisan.
\newblock On {K}hot's unique games conjecture.
\newblock {\em Bulletin (New Series) of the American Mathematical Society}, 49(1), 2012.

\bibitem{Cit:variational_nesterov}
Andre Wibisono, Ashia~C Wilson, and Michael~I Jordan.
\newblock A variational perspective on accelerated methods in optimization.
\newblock {\em Proceedings of the National Academy of Sciences}, 113(47):E7351--E7358, 2016.

\bibitem{ZP08}
M.~M. Zavlanos and G.~Pappas.
\newblock A dynamical systems approach to weighted graph matching.
\newblock {\em Automatica}, 44:2817--2824, 2008.

\bibitem{zhu2022dynamic}
Hongyu Zhu, Stefan Klus, and Tuhin Sahai.
\newblock A dynamic mode decomposition approach for decentralized spectral clustering of graphs.
\newblock In {\em 2022 IEEE Conference on Control Technology and Applications (CCTA)}, pages 1202--1207. IEEE, 2022.

\end{thebibliography}
\end{document}